\documentclass[11pt]{amsart}

\advance\hoffset by -0.5 truecm
\usepackage{amssymb}
\usepackage{amsmath}
\usepackage{graphicx}
\usepackage{color}

\newtheorem{Theorem}{Theorem}[section]
\newtheorem{Lemma}[Theorem]{Lemma}
\newtheorem{Corollary}[Theorem]{Corollary}

\newtheorem{Proposition}[Theorem]{Proposition}

\newtheorem{Remark}[Theorem]{Remark}

\def \dim{{\mbox {dim}}\,}

\def\V{\mbox{Var}}

\def\Z{{\mathbb Z}}
\def\C{{\mathbb C}}
\def\N{{\mathbb N}}

\def\V{{\mathbb V}}

\begin{document}
\title[Negative eingenvalues of the conformal Laplacian] {Negative Eigenvalues of the conformal Laplacian}

\author[G. Henry]{Guillermo Henry}\thanks{G. Henry is partially supported by PICT-2020-01302  grant  from ANPCyT}
\address{Departamento de Matem\'atica, FCEyN, Universidad de Buenos
	Aires and IMAS, CONICET-UBA, Ciudad Universitaria, Pab. I., C1428EHA,
	Buenos Aires, Argentina and CONICET, Argentina.}
\email{ghenry@dm.uba.ar}

\author[J. Petean]{Jimmy Petean}\thanks{J. Petean is supported
	by Fondo Sectorial SEP-CONACYT, grant A1-S-45886.}
\address{CIMAT, A.P. 402, 36000, Guanajuato. Gto., M\'exico}
\email{jimmy@cimat.mx}

\subjclass{53C21}

\date{}

\begin{abstract} Let $M$ be a closed differentiable manifold of dimension at least $3$. 
Let $\Lambda_0 (M)$ be the minimun number of non-positive eigenvalues that the conformal Laplacian of a metric on $M$
can have. 
We prove that for any $k$ greater than or equal to $\Lambda_0 (M)$,  there exists a Riemannian metric on $M$
such that its conformal Laplacian  has exactly $k$ negative eigenvalues.  Also, we  discuss upper  bounds for $\Lambda_0 (M)$.  

\end{abstract}

\maketitle

\section{Introduction}
Let $M$ be  a closed differentiable manifold of dimension $n\geq 3$. We denote with $\mathcal{M}_M$ the space of Riemannian metrics on $M$.  For $g \in \mathcal{M}_M$, the {\it conformal Laplacian} of the Riemannian manifold $(M,g)$, $L_g$, is the linear operator defined by 
\[L_g(u):=\frac{4(n-1)}{n-2}\Delta_gu+s_gu,\]
where  $\Delta_g$  and $s_g$ are the Laplace-Beltrami operator and the scalar curvature, respectively. This operator is interesting from a geometric point of view because it tells us how the scalar curvature changes in a given conformal class. Namely, if $h$ is a pointwise conformal deformation of $g$, i.e., $h=u^{\frac{4}{n-2}}g$ for some $u\in C^{\infty}_{>0}(M)$, then 
\begin{equation}\label{scalar curvature eqn}s_h=L_g(u)u^{-\frac{n+2}{n-2}}. \end{equation}
In particular, it follows from Equation \eqref{scalar curvature eqn} that the  celebrated Yamabe problem, that is the problem of finding a Riemannian metric of constant scalar curvature $c$  in the conformal class of $g$, $[g]$,  is equivalent to find  a positive smooth solution $u$ of  
\[L_g(u)=cu^{\frac{n+2}{n-2}}.\]

The conformal Laplacian is an elliptic and self-adjoint differential operator. Since $M$ is compact, it is well known that the  spectrum of $L_g$  is a non-decreasing unbounded sequence of eigenvalues 
\begin{equation}\label{eigenvalues} \lambda_1(g)<\lambda_2(g)\leq \dots\leq\lambda_k(g) \longrightarrow+\infty 
\end{equation}
where each eigenvalues in \eqref{eigenvalues} appears repeated according to its multiplicity.   
 
The sign of the first eigenvalue has a geometric meaning.  It provides information on the sign that the scalar curvature might have within the conformal class to which the metric belongs. Precisely, in the conformal class $[g]$  there  exists a metric of positive  (negative, zero) scalar curvature if and only if  $\lambda_1(g)$ is positive (negative, zero). This means 
in particular that  the sign of the first eigenvalue of the conformal Laplacian  is a conformal invariant. Actually, it can be seen that sign of the $kth$-eigenvalue is invariant in the conformal class (see \cite{ElSayed}). Therefore,  in a given conformal class $[g]$, we can define the {\it counting map of negative eigenvalues} as  
 \[Neg([g])=\#\big\{ k:\ \lambda_k(g)<0\big\}. \]

Let $k$ be  a non-negative integer.  We denote with $M_k(M)$ the set
\[M_k(M):=\big\{g\in \mathcal{M}_M: Neg([g])=k\big\}.\]

$M_0(M)$ will be the non-empty set if  there exists a metric $g$ with $\lambda_1(g)\geq 0$, which implies that there is a non-negative scalar scalar curvature metric in $[g]$. This is the case of manifolds that admit a metric with positive scalar
curvature, like  the $n-$dimensional sphere $S^n$, or manifolds that admit scalar flat metrics 
but no metric of positive scalar curvature like the  $n-$dimensional torus  $T^n$ (see \cite{Schoen-YauTorus} and \cite{Gromov-Lawson}). However,  $M_0(M)$ might be the empty set for some manifolds. For instance,  a compact quotient of the hyperbolic 3-dimensional space (see \cite{Anderson})  or  a compact K\"{a}hler complex surface with Kodaira dimension $2$ (see \cite{Lebrun2})  do not admit non-negative scalar curvature metrics. 

The smallest integer such that the set $M_k(M)$ is non-empty, $\Lambda(M)$,  is an  invariant of the differential structure of  $M$ that we called the {\it $\Lambda-$invariant}.
$\Lambda(M)$ is the minimum number of negative eigenvalues that a conformal Laplacian can have. In other words,  
\[\Lambda(M):=\min_{g\in \mathcal{M}_M}Neg([g]).\]
If $\Lambda(M)\geq 1$, then $M$ does not admit a metric with non-negative scalar curvature. We can think that
the  $\Lambda-$ invariant measures how far is a manifold from admitting a scalar flat metric. Indeed,  it is easy to see  
 that  $\Lambda(M)=0$ if and only if $M$ admits a  scalar flat metric  (see Section \ref{+Manifolds} for the details).

\medskip

Similarly, we define $\Lambda_0 (M)$ as the minimum of the number of non-positive eigenvalues of conformal Laplacians on
$M$. Clearly $\Lambda (M) \leq \Lambda_0 (M)$. For manifolds which admit metrics of positive scalar curvature they are
both 0, but strict inequality is also possible, for instance $\Lambda(T^n) =0$ and $\Lambda_0 (T^n) =1$.

\medskip

Although there are obstructions to the existence of Riemannian metrics with positive or zero scalar curvature there are no such obstructions for negative scalar curvature metrics (see \cite{Eliasson}, \cite{Kazdan-Warner} and  \cite{Kazdan-WarnerInventione}). This means that in any closed differentiable manifold $M$ of dimension at least $3$ there exists  $g\in \mathcal{M}_M$ such that $\lambda_1(g)<0$. Moreover,  it is also known that for any positive integer $k$ there is a Riemannian metric $g_k$ such that at least the first $k$ eigenvalues of $L_{g_k}$ are negative  (see \cite{ElSayed} and \cite{Y. Canzani}).   

It is natural to consider the question of  whether, for a given manifold $M$ and  a given integer $k$, 
there exists a metric $g\in \mathcal{M}_M$ with exactly $k$ negative eigenvalues. 

The main result of this article is the following theorem:
 
\begin{Theorem}\label{MainThm} Let $M$ be  a closed  connected manifold of dimension at least 3. For any $k\geq \Lambda_0 (M)$,  $M_k(M)$ is not the empty set.

	\end{Theorem}

The article is organized as follows. In Section \ref{Preliminares} we introduce notation, definitions and we review some results from the literature needed to prove Theorem \ref{MainThm}. In Sections \ref{+Manifolds}  and \ref{proofofmainthm} we prove Theorem \ref{MainThm}. In Section \ref{Lambda} we discuss certain bounds for the $\Lambda_0$ and $\Lambda-$invariants.

\section{Preliminaries}\label{Preliminares}

\subsection{Notation and definitions}

Here, we introduce notation and some definitions that will be used throughout the paper.

Given $g\in\mathcal{M}_M$ and $k$ a non-negative integer we can decompose $M_k(M)$ into the disjoint sets 
\[\tilde{M}_k(M):=\{g\in M_k(M):\  \ker(L_g)=0\}\]
and
\[C_k(M):=\{g\in M_k(M):\ \ker(L_g)\neq 0\}.\]

In the following, in order to simplify the notation,  we will omit $M$ from $M_k(M)$,  $\tilde{M}_k(M)$ and $C_k(M)$  when there is no possibility of confusion.

To any closed manifold we associate the sequences of integers  \[s(M)=\{s_i\}_{i\in \N}\ \ \mbox{and}\ \ n(M)=\{n_i\}_{i\in \N}\] 
which satisfy the following properties:
\begin{itemize}
\item $\{s_i\}_{i\in \N}$ is an increasing sequence.

\item $M_{s_i}$ is non-empty for any $i\in \N$.

\item If $M_k$ is non-empty, then there exists  $i$ such that $s_i=k$.

\item $s_{i+1}=s_{i}+n_i$.

\end{itemize}

We will see in Corollary \ref{k0trivial} that there exists $k_0$ such that for any $i\geq k_0$ we have that $n_i =1$.

\subsection{Yamabe constant}

The Einstein metrics are the critical points on $\mathcal{M}_M$ of the  {\it Einstein-Hilbert functional},  defined as 
\[Y(g)=\frac{\int_M s_g dv_g}{Vol(M,g)^{\frac{n-2}{n}}}.\]
  The functional $Y$ is not bounded neither from above nor from below. Actually, it might not have critical points, and
its  critical points are never local extrema, they are saddle points (see \cite{Schoen}).  However, $Y$ is bounded from below in any conformal class and it make sense to  define the so called  {\it 
  	Yamabe constant}
  \[Y(M,[g])=\inf_{h\in [g]}Y(h).\]
  When we restrict the Einstein-Hilbert functional to a conformal class the critical points are constant scalar curvature metrics. 
By the resolution of the Yamabe Problem we known that the Yamabe constant is always achieved (see \cite{Aubin}, \cite{Lee-Parker}, and \cite{Schoen}). Therefore, in each conformal class there exists at least a metric (of a given volumen) with constant scalar curvature.  

It can be seen that within a conformal class it is not possible for two metrics, which do not change the sign of the scalar curvature, to have scalar curvatures with different signs.  In particular, for any metric in $[g]$ with constant scalar, the sign of its scalar  curvature coincides with the sign of $Y(M,[g])$.

The first eigenvalue of $L_g$ satisfies  \begin{equation}\label{firsteigenvalue}\lambda_1(g)= \inf_{u\in C^{\infty}(M)-\{0\}}\frac{\int_ML_g(u)udv_g}{\int_Mu^2dv_g}.\end{equation}
From \eqref{firsteigenvalue} it can be deduced that   $\lambda_1(g)$ and $Y(M,[g])$ have the same sign.  Hence,  if   $\lambda_1(g)=0$ (is positive, negative), then $\lambda_1(h)=0$ (positive, negative) for any $h\in [g]$ and there is a scalar flat metric in $[g]$ (a positive, a negative scalar cuvature metric).   

The {\it Yamabe invariant} of $M$, $Y(M)$, is an important differential invariant, introduced by Kobayashi in \cite{Kobayashi2}  and Schoen in \cite{Schoen}.  It measures the capability of the manifold of admitting a positive scalar curvature metric. $Y(M)$ is defined as  the supremum of all possible Yamabe constants.  We have that  $Y(M)$ is positive if and only if  $M$ carries a metric of positive scalar curvature.

\subsection{Conformal invariance of $L_g$}

The main property of $L_g$ is its conformal invariance. This means that  for a Riemannian metric $h=u^{\frac{4}{n-2}}g$ it holds
\begin{equation}\label{invariance}
L_{h}(v)=u^{-\frac{n+2}{n-2}}L_g(uv).
\end{equation}  
It follows from \eqref{invariance} that if $v$ belongs to kernel of $L_g$ then $L_h(u^{-1}v)=0$.  Therefore,  the dimension of the kernel of the conformal Laplacian is constant along the conformal class. 

Using the min-max characterization of the eigenvalues and the conformal invariance property it can be shown that the sign of $\lambda_i(g)$ does not depends on $g$ but only in $[g]$ (see \cite{ElSayed} for the details).  Hence,  the number of negative eigenvalues of the conformal Laplacian is a conformal invariant as well.

\subsection{Results from the literature}

Here, for the convenience of the reader, we state some important results from
the literature that we are going to use in the next sections to prove Theorem \ref{MainThm}. We refer to \cite{Bar-Dahl} and \cite{Gover_etal} for the proof of Theorem \ref{surgerythm} and \ref{zeroeigenvalue}, respectively.

\subsubsection{Surgery and eigenvalues.}

Let $M$ be a closed $n-$dimensional manifold. Let $S^{k}$ be a $k-$dimensional sphere embedded in $M$ with trivial normal bundle. By  removing a tubular neighborhood $U$ of $S^k$   from $M$  we obtain a manifold with boundary $S^k\times S^{n-k-1}$. The product manifold $D^{k+1}\times S^{n-k-1}$, where $D^{k+1}$ is the closure of the  $k+1-$dimensional ball, has the same boundary as well. Let $W$ be the closed manifold  obtained by plugging toghether $M-U$ and $D^{k+1}\times S^{n-k-1}$ by the common boundary. We say that $W$ is obtained from $M$ by performing surgery of codimension $n-k$, or alternative  by  a $k-$dimensional surgery. The surgery construction is a particular case of a connected sum along submanifolds  (see \cite{Kosinski} for details on the surgery and  the connected sum technique).

The following theorem is due to B\"ar and Dahl \cite{Bar-Dahl}. It  provides us important  information on the behavior of the conformal Laplacian eigenvalues under the surgery procedure.

\begin{Theorem}\label{surgerythm} Let $N$ be a manifold obtained from a closed manifold $M$ by surgery of codimension at least three and let $g\in \mathcal{M}_M$. Then, for any $k\in \N$ and $\varepsilon>0$ there exists $h\in \mathcal{M}_N$ such that 
	\[|\lambda_i(h)-\lambda_i(g)|<\varepsilon\ \  \ \ \ \mbox{for}\  1\leq i \leq k.\]

\end{Theorem}

\subsubsection{The kernel of the conformal Laplacian}

For a generic metric the kernel of  the conformal Laplacian is zero. In a very nice paper, Gover et. al. proved in \cite{Gover_etal} the following result:

\begin{Theorem}\label{zeroeigenvalue} The subset of Riemannian metrics for which zero is not an eigenvalue of the conformal Laplacian is open and dense in $\mathcal{M}_M$.
	\end{Theorem}

\section{Manifolds that admit metrics with non-negative scalar curvature.  }\label{+Manifolds}

In this section we are going to prove that Theorem \ref{MainThm} holds for manifolds that  admit a non-negative scalar curvature metric. More precisely, we are going to show the following proposition:

\begin{Proposition}\label{+M} If $M$ admits a metric with non-negative scalar curvature, then $M_k$ is non-empty for any $k\geq 0$.
	\end{Proposition}

First, let us observe that these three facts are equivalent:
\begin{itemize}
\item[i)] $M$ admits a scalar flat metric.
\item[ii)]  $M$ admits a metric with non-negative scalar curvature.
\item[iii)]	$\Lambda(M)=0$.
\end{itemize}
It is clear that $i)$ implies $ii)$ and that $ii)$ imples $iii)$. Let us assume that $\Lambda(M)=0$.  Then  there exists $g$ with $\lambda_1(g)\geq 0$. Choose any metric $h$ with negative first  eigenvalue.  We denote with  $c_{g,h}:[0,1]\longrightarrow \mathcal{M}_M$  and with  $\lambda_i(t)$ the curve in $\mathcal{M}_M$ 
\begin{equation}\label{curva}c_{g,h}(t)=(1-t)g+th,
\end{equation}
and the $ith-$eigenvalue of $L_{c_{g,h}(t)}$, respectively.  
In \cite{Bar-Dahl}, it was proved that for any $i\geq 1$, the map
\[g\in \mathcal{M}_M\longrightarrow \lambda_i(g)\]
is continuous in the $C^1-$topology.  Given that  $\lambda_1(0)\geq 0$ and $\lambda_1(1)<0$,  there must exists $t_0\in [0,1)$ such that $\lambda_1(t_0)=0$.  Therefore, we can conclude that  in the conformal class of $c_{g,h}(t_0)$ there is a metric with vanishing scalar curvature.

\begin{Remark} The $\Lambda-$invariant, unlike the Yamabe invariant,   does not distinguish between manifolds that  admit scalar flat metrics and  those that admit metrics with positive scalar curvature. On the other hand, the $\Lambda-$invariant can diferentiate manifolds that admits scalar flat metrics from those that only  admit metrics with negative scalar curvature. To give an example,  a compact quotient of a non-abelian nilpotent Lie group $G$, $G/\Gamma$, satisfies that $Y(G/\Gamma)=0$. However, it is well known that $G/\Gamma $ does not admit a  metric with scalar curvature equal to zero. Hence, $\Lambda(G/\Gamma)\geq 1$.
\end{Remark}

In order to prove Proposition \ref{+M} we shall prove the following two lemmas.

\begin{Lemma} $\tilde{M}_1(S^n)$ is non-empty.
	
	\end{Lemma}

\begin{proof} It suffices to prove the existence of a metric $\tilde{g}$ such that  $\lambda_1(\tilde{g})<0$ and $\lambda_2(\tilde{g})>0$. Let us  consider $g$ and $h$  such that $\lambda_1(g)>0$ and  $\lambda_1(h)<0$. Let $c_{g,h}$ be the curve in $\mathcal{M}_{S^n}$ defined  as in \eqref{curva}, and let $t_0$ be the minimum of the set that consists of the values $t$ in  $(0,1)$  with the property that $\lambda_1(t)=0$  and $\lambda_1(s)<0$  in $(t,t+\delta_t)$ for some $\delta_t>0$. Recall that $\lambda_2(t)>\lambda_1(t)$ for any $t\in [0,1]$. Then, by the continuity of the eigenvalues in the $C^1-$topology, we have that for $\delta>0$
small enough,  $\lambda_1(t_0+\delta)<0$ and $\lambda_2(t_0+\delta)>0$. 
\end{proof}

\begin{Lemma}\label{nogap}  If $\tilde{M}_k(M)$ is not the empty set, then $\tilde{M}_{k+1}(M)$ is non-empty.
	
\end{Lemma}

\begin{proof} Let $h\in \tilde{M}_k(M)$ and let $g\in \tilde{M}_1(S^n)$ ($n=\dim(M)\geq 3$).   Let us consider the disjoint union $M \amalg S^n $ endowed with the Riemannian metric $h\amalg g$. The spectrum of $L_{h\amalg g}$ is the disjoint union of the spectra of $L_h$ and $L_{g} $.   Since $\lambda_k(h)<0$ and  $\lambda_{k+1}(h)>0$, the spectrum of $L_{h\amalg g}$ has exactly $k+1$ negative eigenvalues and $\ker(L_{h\amalg g})	=\{0\}$. Let $(p,q)\in M \amalg S^n$. We perfom a $0-$dimensional surgery at $(p,q)$, that is, we obtain the connected sum $M \#S^n$, which is diffeomorphic to $M$. By Theorem \ref{surgerythm}, given $\varepsilon>0$  there exists $g_{\varepsilon}\in \mathcal{M}_{M\#S^n} $ such that
	\[|\lambda_i(g_{\varepsilon})-\lambda_i(h\amalg g)|<\varepsilon\]
	for any $1\leq i\leq k+2$. Therefore, for $\varepsilon$ small enough, $g_{\varepsilon}\in \tilde{M}_{k+1}(M \# S^n)$.   Let  $\varphi:M \longrightarrow M \# S^n$ be a diffeomorphism, then  $\varphi^{*}(g_{\varepsilon})\in \mathcal{M}_M$ is isometric to $g_{\varepsilon}$. Hence,  $\varphi^{*}(g_{\varepsilon})\in \tilde{M}_{k+1}(M)$.
	
\end{proof}

 Take a  Riemannian metric $g$ on $M$ such that for some $k\geq 1$, $\lambda_k(g)<0$. By Theorem \ref{zeroeigenvalue}, there exists, as close as we want  to $g$,  a metric $g_*$ with at least $k$ negative eigenvalues which conformal Laplacian has trivial kernel, that is $g_*\in \tilde{M}_{r}(M)$ for some $r\geq k$.   By applying the Lemma \ref{nogap}  successively, we obtain that $\tilde{M}_{r+l}(M)\neq \emptyset$  for any $l\geq 0$. We have proved the following corollary. 

\begin{Corollary}\label{k0trivial} For any closed manifold $M$, there exists $k_0$ such that $M_k (M) \neq \emptyset$ for any
$k\geq k_0$. 
	\end{Corollary}

\begin{proof}[Proof of Proposition \ref{+M}] By asumption $M_0\neq \emptyset$.  If $M$ admits a metric of positive scalar curvature,  then $\tilde{M}_0$ is non-empty and  by Lemma \ref{nogap} we have that $\tilde{M}_k\neq \emptyset\ \mbox{for any}\ k\geq 0.$
	
	Let assume that $M_0=C_0$. Let $g_0$ be  a scalar flat metric on $M$. The  functional $Y$   restricted to the conformal class $[g_0]$ attains a minimum at $g_0$.  Since $Y$ does not have local extrema, there exists a traceless symmetric $(0,2)-$type tensor $H$  with $\delta(H)=0$ such that the curve $c(t)=g_0+tH$ satisfies 
	\[Y(c(t))<0\ \mbox{for}\ |t|\in (0, \varepsilon),\]
	(see for instance \cite{Schoen}). 
	
	The sequence of metrics $c(t_n)$ converges to $g_0$ when $t_n$ tends to $0$.   Then, 
	\[\lambda_2(c(t_n))\longrightarrow \lambda_2(g_0)>\lambda_1(g_0)=0.\]
Consequently, for sufficiently large $n$, we have that $\lambda_1(g(t_n))<0$ and  $\lambda_2(g(t_n))>0$, namely  $g(t_n)\in \tilde{M}_1$. Finally,  the proposition follows by repeatedly applying Lemma \ref{nogap}.
	
\end{proof}

\section{Proof of Theorem \ref{MainThm}}\label{proofofmainthm}

To complete the proof of Theorem \ref{MainThm} it remains to address the case when $\Lambda(M)\geq 1$.

\begin{Proposition}\label{prop n=1} Assume that $\Lambda(M)\geq 1$. If $g_0$ is a metric on $M$ such that
$L_{g_0}$ has exacly $\Lambda_0 (M)$ non-positive eigenvalues then 0 is not an eigenvalue of $L_{g_0}$.
	
	\end{Proposition}

\begin{proof}  Assume that $\dim(\ker(L_{g_0}))=m_0 \geq 1$. Let $n_0$ be the number of negative eigenvalues of
$L_{g_0}$, so that $\Lambda_0 (M)= n_0 + m_0$.
 Let $g(t)$ be an analytic curve in $\mathcal{M}_M$ such that $g(0)=g_0$. 
Let $\lambda_i (t) = \lambda_i (g(t))$ be the eigenvalues of $L_{g(t)}$.

Note  that 
\[\lambda_1(0)<\lambda_2(0)\leq \cdots\leq\lambda_{n_0}(0)<0,\]
\[\lambda_{n_0 +1}(0)=\cdots =\lambda_{n_0 + m_0 }(0)=0,\]
and 
\[\lambda_{n_0 +m_0 +1}(0)>0.\]

We claim that there is an interval $I$ centered at $t=0$  with the property that
\begin{align*}
\lambda_i(t)<0 &\ \mbox{if}\  1\leq i\leq n_0,\ \mbox{and}\ t\in I, \\
\lambda_i(t)\leq 0&\  \mbox{if}\ n_0+1\leq i\leq n_0+m_0 \ \mbox{and}\ t\in I.
\end{align*} 
Indeed,  by continuity of the eigenvalues, there exists $I$ centered at  $t=0$ such that for each $t\in I$, the  first $n_0$ eigenvalues of $L_{g(t)}$ are negative and $\lambda_{n_0 + m_0 +1}$ is positive.    
Assume that for some $1\leq l \leq m_0$, there is  $t_0\in I$ with $\lambda_{n_0+l}(t_0)>0$. Then  $g(t_0)$ has at most $n_0+l-1
< \Lambda_0 $ non-positive eigenvalues, which is a contradiction by the definition of $\Lambda_0 (M)$.

\medskip

Let $H$ be  any symmetric tensor of type $(0,2)$. Let $g(t)$ be an analytic curve such that $g(0)=g_0$ and $g'(0)=H$.  We 
denote by $\pi_{\V_0}$ the $L^2$-orthogonal  projection into $\V_0=\ker(L_{g_0})$. And we consider,
as in \cite{Gover_etal}, the linear operator $Q_{H}: \V_0\longrightarrow \V_0$ defined by 
\[Q_{H}(u):=\pi_{V_0}\circ (L_{g})'(0).\]

Recall that $m_0 = \dim ( \V_0 )$. It is known that there are analytic functions $\mu_1 , \dots , \mu_{m_0}$ defined
on an interval $(-\varepsilon , \varepsilon )$ such that for $t \in (-\varepsilon , \varepsilon )$, 
$\mu_1 (t) , \dots , \mu_{m_0} (t)$ are the $m_0$ eigenvalues of $L_{g(t)}$ which are close to 0
(see for instance the discussion in \cite{Gover_etal}). Let $f_i$ be a non-zero eigenfunction corresponding to $\mu_i$, 
which also depends analytically on $t$. By a direct computation
if follows that $\mu_i ' (0)$, $ i=1,...,m_0$, are eigenvalues of $Q_H$ (with corresponding eigenfunction $f_i$). 
The previous comments say that for any $H$ the funcions $\mu_i$ attained a local maximum at $0$ and therefore
$Q_H \equiv 0$.

On the other hand it is proved in (Proposition 4.1, \cite{Gover_etal}) that when the Yamabe constant of $[g_0 ]$ is negative
(which is our case since we are assuming that $\Lambda (M) \geq 1$) we have:
\begin{center}
{\it There exists  a $(0,2)$ symmetric tensor  $H$ such that $Q_{H}\neq 0$.}
\end{center} 
This is of course a contradiction and we have proved the proposition.  

\end{proof}

\begin{proof}[Proof of Theorem \ref{MainThm}] 	If $\Lambda(M)=0$  the theorem  follows from Proposition \ref{+M}. If $\Lambda(M)\geq 1$  it follows from Proposition \ref{prop n=1} that $\tilde{M}_k(M)\neq\emptyset$ if $k =\Lambda_0 (M)$.
Then the theorem follows by applying again Lemma \ref{nogap}.
\end{proof}

\begin{Lemma}\label{sizekernel} If $n_k\geq 2$ then 
	\begin{equation*}
	\dim(\ker(L_g))\geq n_k
	\end{equation*}
	for any $g\in M_{s_k}(M)$.
\end{Lemma}
\begin{proof}
	Let us assume the contrary. Let  $g\in M_{s_k}$ with  $\dim(\ker{L_g})< n_k$. Given $\varepsilon$, we obtain a metric $h$  as close to $g$ as we want, such that  the number of eigenvalues of  $L_h$ in the interval  $(-\varepsilon,\varepsilon)$  is $\dim(\ker{L_g})$ and $\ker{L_h}=\{0\}$.  It follows that 
	\[s_k \leq Neg(h)\leq s_k+\dim(\ker{L_g})< s_k+n_k=s_{k+1}.\] 
	Then, necessarily $Neg(h)=s_k$. Therefore $h\in \tilde{M}_{s_k}$ and  by Lemma \ref{nogap} we have that $s_{k+1}=s_k+1$ which is a  contradiction.
\end{proof}

\begin{Proposition} Let assume that there exists $g_0\in M_{s_k}$ such that $\dim(\ker( L_{g_0}))=n_k$, then $n_k=1$.
	\end{Proposition}
\begin{proof} By Proposition \ref{+M} we may assume that $\Lambda(M)\geq 1$ and $n_k\geq 2$.  Let $g(t)$ be an analytic curve such that $g(0)=g_0$. There is an interval $I$ centered at $0$  such that $g(t)$ has at least $s_k$ negative eigenvalues for any $t\in I$.  Applying Lemma \ref{sizekernel},  we have  that $\lambda_{i}(t)\leq 0$ if $s_k+1\leq i\leq s_k+n_k$ and $t\in I$.  However, using  a similar argument to the one used in the proof of Proposition \ref{prop n=1} we obtain a contradiction.
	\end{proof}

\section{$\Lambda-$invariant}\label{Lambda}

In this section, we will discuss upper bounds for the $\Lambda-$invariant. In \cite{Bar-Dahl}, B\"{a}r and Dahl introduced the $\kappa-invariant$ that measures how big is the almost kernel of the conformal Laplacian. Precisely, $\kappa(M)$ is the smallest non-negative integer such that for any  $m\in \N$ there is $g_m\in \mathcal{M}_M$ that satisfies that $|\lambda_i(g_m)|\leq 1$  for $1\leq i< \kappa(M)$ and   
$\lambda_{\kappa(M)}(g_m)>m$. Clearly, $\kappa(M)$ is greater or equal than $\Lambda_0 (M)
\geq \Lambda (M)$. Hence, the upper bounds obtained in \cite{Bar-Dahl} for   $\kappa(M)$ hold  for $\Lambda(M)$.

For non-connected manifolds, the $\Lambda-$invariant might be as big as we want. Let $M$  with $Y(M)<0$, then $\Lambda(M)\geq 1$.   Given a positive integer  $k$, let us consider  the disjoint union  $W=M\sqcup\dots \sqcup M$ of $k$ copies of $M$.   Since the spectrum of the conformal Laplacian of a disjoint union of Riemannian manifolds is the union of their  spectra, then \[\Lambda (W)=k\Lambda(M)\geq k.\]

\subsection{$\Lambda-$invariant and surgery.}

In this subsection we are going to discuss the behavior of the $\Lambda-$invariant under surgeries. 

Let \[z(M):=\inf\big\{\dim(\ker(L_g)):\ g\in M_{\Lambda(M)}\big\}.\]

\begin{Remark}
We have that $\Lambda (M) + z(M) \geq \Lambda_0 (M)$. If $\Lambda (M) + z(M) =  \Lambda_0 (M)$ then
$z(M)=0$, by Proposition 4.2.
\end{Remark}

\begin{Proposition}\label{surgeryLemma}
		Let $M$ be a closed manifold. If $N$ is obtained from $M$ by performing surgery of codimension at least $3$, then 
                    \[ \Lambda (N) \leq \Lambda_0 (N)\leq \Lambda_0 (M) \leq \Lambda (M) + z(M).\]
		\end{Proposition}

\begin{proof} We have to prove that $ \Lambda_0 (N)\leq \Lambda_0 (M)$ Let $g$ be a metric on $M$ 
that attains $k=\Lambda_0 (M)$. Then  $\lambda_k (g) <0$ and $\lambda_{k+1}(g)>0$
(by Proposition 4.2). By Theorem \ref{surgerythm}, given $\varepsilon>0$ there is $h_{\varepsilon}\in \mathcal{M}_N$ such that $|\lambda_i(h_{\varepsilon})-\lambda_i(g)|<\varepsilon$ for any $i\leq k+1$. Therefore, for $\varepsilon$ small enough, we have that $L_{h_{\varepsilon}}$ has exactly $k$ non-positive (which are actually negative) eigenvalues.
	\end{proof}

\begin{Corollary} Let $M$ be a closed connected manifold. Assume that $N$ is obtained from $M$ by performing surgery of codimension at least $3$. If $\Lambda (M) = \Lambda_0 (M)$ then $\Lambda(N)\leq \Lambda(M) $. If $\Lambda (M) =0$ then
$\Lambda(N) =0$ or $\Lambda (N) =1$.
	\end{Corollary}

\begin{proof} The first statement follows directly from the previous proposition. If $\Lambda (M)=0$ then $M$ admits a
scalar flat metric. If $M$ admits a positive scalar curvatre metric then it is well known that $N$ also admits a positive scalar
curvature metric and $\Lambda (N)=0$. Even if $M$ does not admit a metric of positive scalar curvature it is shown in the
proof of Proposition 3.1 that $\tilde{M}_1 \neq \emptyset$. Then it follows as before from Theorem \ref{surgerythm} then
$\tilde{N}_1 \neq \emptyset$ and therefore $\Lambda (N) \leq 1$.
	\end{proof}

 A $k-$dimensional  ($k\geq 1$) surgery  on a $n-$dimensional manifold can be undone by performing a $n-k-1-$dimensional surgery.  If we want to  reverse  a $k-$dimensional surgery by doing  surgery of codimension at least $3$,  then $ n-k-1\leq n-3$.  Therefore, any $k-$dimensional surgery with $2\leq k \leq n-3$  can be undo by a surgery of codimension at least 3. Is $M$ is a connected manifold, a $0-$dimensional surgery can be reverted by  a 1-dimensional surgery (see for instance \cite{Petean2}). 
The following result then follows from Proposition 5.2:
 
\begin{Proposition}\label{EqualityLambda} Let $N$ be a manifold obtained from a connected $n-$dimensional  closed manifold $M$ by performing a $k-$dimensional surgery. Assume that   
$k\leq n-3$,  $k\neq 1$ and $n\geq 4$. Then 
\begin{equation}\label{equalityL}\Lambda_0 (M)=\Lambda_0 (N).
\end{equation}
	\end{Proposition}

\subsection{Upper bounds}
 For  a positive integer $k$ we denote with $M^k$ the product of $k$ copies of $M$.  We write with $M^0$  the one point manifold;  with  $M+N$ we denote indistinctly  the connected sum or the disjoint union of $M$ and $N$; with $kM$ we denote the disjoint union or the connected sum of $k$ copies of $M$ ; and with  $0.M$  the $n-$dimensional sphere. 
 
 \subsubsection{Spin cobordism groups.} Here,  we recall some facts on  the spin cobordism groups and the $\alpha-$genus. We refer the reader to \cite{Lawson-Michelson} and the references therein for a detailed treatment of these subjects. Let $M$ and $N$ be  two closed spin $n-$dimensional manifolds.  We say that $M$ and $N$ are spin cobordant, $M\sim N$,  if  there exists a spin  $n+1-$dimensional manifold $W$  with boundary   $M\sqcup -N$. The {\it n-spin cobordism group}, $\Omega^{Spin}_n$, is the quotient of the  space of spin $n-$dimensional manifolds by the latter equivalence relation. 
The $\alpha$-genus (see \cite{Milnor})  is a suryective homomorphism  between $\Omega^{Spin}_n$ and the real $K-$homology group of a point $K_nO(pt)$. Indeed, two spin manifolds  are cobordant if and only if they have the same $\alpha-$genus. The $\alpha$-genus gives an obstruction to the existence of positive scalar curvature metrics. By the works of  Lichnerowicz \cite{Lich} and Hitchin \cite{H}  we know that  for a manifold with positive Yamabe invariant the    $\alpha-$genus must vanish. Moreover, it was shown by Stolz in \cite{Stolz} that if $M$ is simply connected of dimension at least 5 and $\alpha(M)=0$ then $Y(M)>0$.    The group $K_nO(pt)$ is  trivial  for  $n\equiv 3,5,6,7$ mod. $8$;  $K_nO(pt)=\Z_2$ when $n\equiv 1,2$ mod. $8$; and $K_nO(pt)=\Z$ if $n\equiv 0,4$ mod. $8$. Therefore,  $Y(M)>0$  if $M$ is a simply connected spin manifold of dimension $n\equiv 3,5,6,7$ mod. $8$ of dimension at least $5$.  In \cite{Petean}, it is  showed that for any simply connected manifolds of dimension at least $5$, the Yamabe invariant is not negative. In order to stablish Proposition \ref{upperboundP} we use a similar argument to the one used in \cite{Petean}.

 The Joyce manifold, $J_8$,  is a  simply connected closed spin  $8-$dimensional Riemannian manifold with holonomy group  $Spin(7)$ (see \cite{Joyce}). $J_8$ is Ricci flat and satisfies $\alpha(J_8)=1$. Therefore,   $\alpha(J_8^{k})$ generates $K_{8k}O(pt)$.  The Kummer surface, $K3$, is the hyperplane in $\C P^3$ defined  by $K3:=\{[z_0:z_1:z_2:z_3]\in \C P^3:\  z_0^2+z_1^2+z_2^2+z_3^2=0\}$. $K3$ is a closed simply connected spin manifold of dimension $4$, that carries a Ricci flat metric but does not admit a positive scalar curvature metric. Its $\alpha-$genus is 1. Hence, $J_8^{k}\times lK3$ is a closed spin $8k+4-$dimensional manifold that admits a scalar flat metric with $\alpha(J_8^{k}\times lK3)=l\in K_{8k+4}O(pt)$. On the other hand, 
$\alpha(J_8^k\times S^1)=1$  and $\alpha(J_8^k\times S^1\times S^1)=1$ generate $K_{8k+1}O(pt)$ and $K_{8k+2}O(pt)$, respectively. 
 
 From the comments above we have that for any  $n$ and $s\in K_nO(pt)$ there exists a closed spin manifold $N_s$ that admits a scalar flat metric and $\alpha([N_s])=s$. Note that if $n\equiv 1,2,3,5,6,7$ mod. $8$ or $n\equiv 0,4$ mod. 8 and $s=0,1$, $N_s$  can be chosen connected. 
But if  $s\geq 2$,   $N_s$ is not be connected.   More precisely,  we have that $N_s=s.N$ where $N$ is a connected spin manifold that carries a scalar flat metric but does not admit a positive scalar metric. Note that  $z(N_s)=s$.

\begin{Lemma}\label{lemma2} Each class of $\Omega^{Spin}_n$ contains a  manifold with vanishing $\Lambda-$invariant.
	\end{Lemma}

\begin{proof} If $n\equiv 3,5,6,7$ mod. $8$, any spin manifold is cobordant to $S^n$.     Hence, let us  assume that $n\equiv 0,1,2,4$ mod. $8$.  Let $M$ be a closed spin manifold and  let $N_{\alpha(M)}$ as in the paragraph above. Then, $\alpha(M-N_{\alpha(M)})=\alpha(M)-\alpha(N_{\alpha(M)})=0$.  Therefore,   there exists a simply connected manifold $S$ with positive Yamabe invariant that is cobordant with $M-N_{\alpha(M)}$ (see \cite{Stolz}). This implies that
	 $M\sim S+N_{\alpha(M)}$. Since $N$ admits a scalar flat metric,  $S+N_{\alpha(M)}$ carries a metric $g$  with $\lambda_1(g)\geq 0$ and this completes the proof.
\end{proof}

\subsubsection{Upper bound}

Let  $D(M)$ be $\alpha(M)$  if $M$ is a closed spin manifold  and $0$ if $M$ is not spin. The following Proposition was obtained for the $\kappa-$invariant by B\"{a}r and Dahl in \cite{Bar-Dahl}.

\begin{Proposition}\label{upperboundP} If $M$ is a closed simply connected manifold of dimension $n\geq 5$  then 
	\[\Lambda(M) \leq \Lambda_0 (M) \leq |D(M)|.\]
	\end{Proposition}

\begin{proof}
	
	Gromov and Lawson proved in \cite{Gromov-Lawson} that any non-spin simply connected manifold of dimension at least $5$ admits a positive scalar curvature metric. Therefore, the  $\Lambda-$invariant vanishes for these manifolds. Hence, let us assume that $M$ is a closed simply connected spin manifold of dimension at least 5.  If  $\alpha(M)=0$, then $M$ admits a metric of positive scalar curvature and $\Lambda(M)=0$. This always happens when $n\equiv 3,5,6,7$ mod. 8.  For this reason suppose  that $n\equiv 0,1,2,4$ mod. $8$ and $\alpha(M)\geq 1$. By Lemma \ref{lemma2}, $M$ is spin cobordant  with  a manifold  $N$  such that $\Lambda(N)=0$. From the  proof of lemma  $N= S+N_{\alpha(M)}$, where $Y(S)>0$, $N_{\alpha(M)}$ admits a scalar flat metric and satisfies that  $z(N_{\alpha(M)})=\alpha(M)$. Hence, $z(S+N_{\alpha(M)})=\alpha(M)$.  Since $M$ is simply connected, $M$ can be obtained from any manifold that belongs to its cobordant class  by performing surgery of codimension at least 3 (see \cite{Gromov-Lawson}).  Applying  Proposition \ref{surgeryLemma} we have that \[\Lambda(M)\leq \Lambda(N)+z(N_{\alpha(M)})=\alpha(M),\]
which  finishes the proof.
\end{proof}

As an inmediate consequence  we obtain the following corollary.

\begin{Corollary} Let $M$  be a  closed simply connected manifold of dimension at least 5. If either $M$ is not spin or $M$ is spin and $n\equiv 1,2,3,5,6,7$ mod. 8 or $n\equiv 0,4$ mod. 8 with $|\alpha(M)|\leq 1$, then $\Lambda_0 (M)\leq 1$.
	\end{Corollary}

We suspect that for any closed simply connected manifold of dimension at least $5$ the  $\Lambda_0-$invariant  is not greater than 1.

\end{document}